\documentclass{mfat}
\pagespan{137}{151}

\usepackage{mmap}
\usepackage[utf8]{inputenc}

\newtheorem{theorem}{Theorem}
\newtheorem{prop}{Proposition}
\newtheorem{cor}{Corollary}

\newtheorem{deff}{Definition}

\begin{document} 

\title[On nonsymmetric  rank one singular perturbations $\ldots $]
    {On nonsymmetric  rank one singular perturbations  of selfadjoint operators}

\author{Mykola Dudkin}
\address{National Technical University of Ukraine "Kyiv Polytechnic Institute", 37 Prospect Pe\-re\-mo\-gy,
Kyiv, 03056, Ukraine}
\email{dudkin\@imath.kiev.ua}

\author{Tetiana Vdovenko}
\address{National Technical University of Ukraine "Kyiv Polytechnic Institute", 37 Prospect Pe\-re\-mo\-gy,
Kyiv, 03056, Ukraine}
\email{tanyavdovenko@meta.ua}
%
\date{22/09/2015; \ \  Revised 23/02/2016}
\begin{abstract}
  We consider nonsymmetric rank one singular perturbations of a
  selfadjoint operator, i.e., an expression of the form $\tilde
  A=A+\alpha\left\langle\cdot,\omega_1\right\rangle\omega_2$,
  $\omega_1\not=\omega_2$, $\alpha\in{\mathbb C}$, in a general case
  $\omega_1,\omega_2\in{\mathcal H}_{-2}$.

  Using a constructive description of the perturbed operator $\tilde
  A$, we investigate some spectral and approximations properties of
  $\tilde A$. The wave operators corresponding to the couple $A$,
  $\tilde A$ and a series of examples are also presented.
\end{abstract}

\maketitle

\section{Introduction}

The theory of rank one (symmetric) singular perturbations of a self
adjoint operator has obtained much attention of physicists and
mathematicians.  Several papers and research monographs are devoted to
this theory (see, f.e., \cite{AGHKH,AKu1,Ko17} and references
therein).

The aim of the current paper is to investigate a generalization of the
singular (symmetric) perturbation theory to the case of nonsymmetric
perturbations of the form $\tilde
A=A+\alpha\left\langle\cdot,\omega_1\right\rangle\omega_2$, where
$A=A^*$ is a given selfadjoint operator perturbed by
$\alpha\left\langle\cdot,\omega_1\right\rangle\omega_2$ with vectors
$\omega_1\not=\omega_2$ that belong to the negative space ${\mathcal
  H}_{-2}$ from the $A$-scale (the scale is generated by the operator
$A$) and $\alpha\in{\mathbb C}$.

If $\omega_1=\omega_2$ and $\alpha\in{\mathbb R}$, then  we meet the classical case, i.e.,
the well known theory of a
(symmetric) singular perturbation of selfadjoint operators \cite{AGHKH,AKu1,Ko17}.

In this article we continue our investigations started in \cite{VD},
where we considered only the case $\omega_1,\omega_2\in{\mathcal
  H}_{-1}$.  A significant improvement of the previous studies is in
the consideration of perturbations by arbitrary vectors, i.e., vectors
from ${\mathcal H}_{-2}$, too.

Since the study of singularly perturbed operators is extended to 
perturbations by nonsymmetric potentials, we will also expect that the
spectral properties (in particular the point spectrum of $\tilde A$)
must be similar as in the classical case (including the unexpected for the perception the
so-called associated pair of eigenvalues).

The next important topic of this article is to clarify 
approximation properties of these operators.  In particular, we
investigate an approximation of classical (symmetric) perturbations
by nonsymmetric perturbations taking into account $H_{-1}$- and
$H_{-2}$-perturbations.

When investigating wave operators for the couple $A$ and a
nonsymmetrically perturbed $\tilde A$, we solve a number of problems,
including the following: whether $\tilde A$ is a spectral type
operator; whether $\tilde A$ is additionally of a scalar type one;
whether there exist wave operators for the couple $A$ and $\tilde A$;
and whether we can write explicit expressions for these wave operators
and, as a consequence, the scattering matrix.

In general, the idea and motivations for considering nonsymmetric
perturbations is not new.  Closely related investigations are in
\cite{L,MM}, the ones carried out from the point of view of
nonselfadjoint extensions in \cite{V}, and a non-local interactions
approach has been taken in \cite{N1,N2}.

\section{The main definitions and properties}

Let ${\mathcal H}$ be a separable Hilbert space with a scalar product
$(\cdot,\cdot)$ and the norm $\Vert\cdot\Vert=\sqrt{(\cdot,\cdot)}$.
We consider a selfadjoint semi-bounded operator $A=A^*$ defined on a
domain ${\rm Dom} A={\mathfrak D}(A)$ in ${\mathcal H}$.  The sets
$\sigma(\cdot)$, $\sigma_p(\cdot)$, $\sigma_{ac}(\cdot)$,
$\sigma_c(\cdot)$, $\rho(\cdot)$ denote the spectrum, the point
spectrum, the absolutely continuous spectrum, the continuous spectrum,
and the regular points of a corresponding operator, respectively.

An operator $A$ is associated with the $A$-scale of Hilbert spaces
\cite{AKu1}.  We consider only a part of the $A$-scale, namely,
\begin{equation}\label{f1}
{\mathcal H}_{-2}\supset{\mathcal H}_{-1}\supset{\mathcal H}\equiv{\mathcal H}_0
\supset{\mathcal H}_{+1}\supset{\mathcal H}_{+2}
\end{equation}
where ${\mathcal H}_{+1}:={\mathfrak D}(\vert A\vert^{1/2})$ and
${\mathcal H}_{+2}:={\mathfrak D}(A)$ are endowed
with the  norms
$\Vert\varphi\Vert_{k}=\Vert(\vert A\vert+I)^{k/2}\varphi\Vert$, $k=1,2$, $\varphi\in{\mathcal H}_k(A)$,
respectively  ($I$ stands for the identity);
and ${\mathcal H}_{-k}:={\mathcal H}_{-k}(A)$ is the negative (dual) space, i.e., the completion of
${\mathcal H}$ with respect to the norm
$\Vert f\Vert_{-k}=\Vert(\vert A\vert+I)^{-k/2}f\Vert$, $k=1,2$, $f\in{\mathcal H}$.
Let $\langle\cdot,\cdot\rangle$ denote the usual dual scalar product for the spaces
${\mathcal H}_k$ and ${\mathcal H}_{-k}$.
The inner product in ${\mathcal H}_k$ and ${\mathcal H}_{-k}$ is denoted by $(\cdot,\cdot)_{\pm k}$,
$k=1,2$.

The operator $A$ has an extension by continuity to ${\mathcal H}$
(${\mathcal H}_{1}$) and it is understood as a bounded operator
from ${\mathcal H}$ (${\mathcal H}_{1}$) into ${\mathcal H}_{-2}$
(${\mathcal H}_{-1}$).  We denote such an extension by ${\bf A}$ and
${\bf R}_z=({\bf A}-z)^{-1}$, $z\in\rho({\bf A})$, is a corresponding
resolvent.

In some cases, we can continue the usual dual scalar product
$\langle\cdot,\cdot\rangle$ to the case $\langle\omega,\phi\rangle$,
where $\omega,\phi\in{\mathcal H}_{-2}$ (of course $\omega\not=\phi$)
in the following way.  For example, if we can decompose vectors
$\omega=\omega_1+\omega_2$ and $\phi=\phi_1+\phi_2$ so that ${\rm
  spsupp}(\omega_i)\subseteq\Pi_i$, ${\rm
  spsupp}(\phi_i)\subseteq\Pi_i$, $i=1,2$, $\Pi\cap\Pi=\varnothing$
and $\omega_1\in{\mathcal H}_{+2}$, $\omega_2\in{\mathcal H}_{-2}$,
$\phi_1\in{\mathcal H}_{-2}$, $\phi_2\in{\mathcal H}_{+2}$, then we
can have
$\langle\omega,\phi\rangle=\langle\omega_1,\phi_1\rangle+\langle\omega_2,\phi_2\rangle<\infty$,
and $\langle\omega_1,\phi_2\rangle=\langle\omega_2,\phi_1\rangle=0$,
since ${\rm spsupp}(\omega_1)\cap{\rm spsupp}(\phi_2)=\varnothing$ and
${\rm spsupp}(\omega_2)\cap{\rm spsupp}(\phi_1)=\varnothing$.  Here
${\rm spsupp}(\cdot)$ denotes the spectral support of the
corresponding vector in the sense of the operator ${\bf A}$. By
definition (cf. \cite{D0}), for $\omega\in{\mathcal H}_{-2}$,
\begin{multline*}
{\rm spsupp}(\omega):=\{
\lambda\in{\mathbb R} \ \vert \ \forall O_{\lambda,\varepsilon} \exists
\psi\in C_0({\mathbb R})\cap L_2({\mathbb R},d\rho(\lambda)) \ : \\
{\rm supp}(\psi)\subset O_{\lambda,\varepsilon} \ \text{and} \
\int \limits_{{\mathbb R}} \hat\omega(\lambda)\psi(\lambda)\,d\rho(\lambda)\not=0
\},
\end{multline*}
where $O_{\lambda,\varepsilon}$ is an $\varepsilon$-neighborhood of a
point $\lambda$; $C_0({\mathbb R})$ is the set of continues functions
with compact supports on ${\mathbb R}$; $\hat\omega(\lambda)$ denotes
the Fourier image of the vector $\omega$. According to the central
spectral theorem \cite{B3}, the Fourier transform between ${\mathcal
  H}_{-2}$ and $L_2({\mathbb R},d\rho(\lambda))$ takes the operator
${\bf A}$ to the multiplication operator by the independent variable
$\lambda$ in the space $L_2({\mathbb R},d\rho(\lambda))$ with the
Borel measure $d\rho(\lambda)$.

Let us consider an operator $V$ in the $A$-scale, such that
${\mathfrak D}(V)\subseteq{\mathcal H}_{+k}$ and ${\mathfrak
  R}(V)\subseteq{\mathcal H}_{-k}$, $k=1,2$.  In our case,
$V=V^{\omega_1,\omega_2}=\langle\cdot,\omega_1\rangle\omega_2$,
$\omega_1, \omega_2\in{\mathcal H}_{-k}$, $k=1,2$.  Since the operator
${\bf A}$ is bounded and acts from ${\mathcal H}_0$ into the whole
${\mathcal H}_{-k}$, $k=1,2$, the expression ${\bf A}+V$ is a bounded
linear operator from ${\mathcal H}_{+k}$ into ${\mathcal H}_{-k}$. Let
us remark that due to \cite{BB} the adjoint operator $({\bf A}+V)^+$
is correctly defined and acts also from ${\mathcal H}_{+k}$ into
${\mathcal H}_{-k}$, $k=1,2$.

Now, the  formal expression
$A+\alpha\langle\cdot,\omega_1\rangle\omega_2$ has a sense of an
operator ${\bf A}+\langle\cdot,\omega_1\rangle\omega_2$ defined on
${\mathcal H}_{+k}$, acting from ${\mathcal H}_{+k}$ into ${\mathcal
  H}_{-k}$, $k=1,2$, and restricted to ${\mathcal H}$,
\begin{equation}\label{f2}
A^{\omega_1,\omega_2}=({\bf A}+\langle\cdot,\omega_1\rangle\omega_2)\restriction_{{\mathcal H}}.
\end{equation}

In what follows, we write usually $A$ instead of ${\bf A}$
and hence, $R_z$ will be used instead of ${\bf R}_z$.

Without loss of generality, supposing also that the operator $A$ is
strongly positive $A>0$, we give a constructive definition of the
operator $A^{\omega_1,\omega_2}$.

\begin{deff}\label{d1}
  Let $A$ be a positive self-adjoint operator on a separable Hilbert
  space ${\mathcal H}$, and consider $A^{\omega_1,\omega_2}$ in
  (\ref{f2}) with $\omega_i\in{\mathcal H}_{-2}\setminus{\mathcal H}$,
  $\omega_1\not=\omega_2$. We put $\eta_i={\bf A}^{-1}\omega_i$,
  $i=1,2$.

  I. The operator $A^{\omega_1,\omega_2}$ is called {singularly
    rank-one nonsymmetrically \emph{uniquely}} perturbed with
  respect to the operator $A$ if $\vert (A^{1/2}\eta_2,A^{1/2}\eta_1)
  \vert<\infty$.  Moreover,
\begin{equation}\label{f3}
{\mathfrak D}(A^{\omega_1,\omega_2})=
\left\{\psi=\varphi+b\eta_2 \ \vert \ \varphi\in{\mathfrak D}(A), \
b=b(\varphi)=\frac{(A\varphi,\eta_1)}{1+(A^{1/2}\eta_2,A^{1/2}\eta_1)}\right\}
\end{equation}
in the case $(A^{1/2}\eta_2,A^{1/2}\eta_1)\not=-1$, and
\begin{equation}\label{f4}
{\mathfrak D}(A^{\omega_1,\omega_2})={\mathfrak D}_{{\mathcal H}_1}\dot +\{c\eta_2\}, \quad
{\mathfrak D}_{{\mathcal H}_1}=\left\{\varphi\in{\mathfrak D}(A) \ \vert \
(A\varphi,\eta_1)=0 \right\}
\end{equation}
in the case $(A^{1/2}\eta_2,A^{1/2}\eta_1)=-1$, (this fact we denote
by $A^{\omega_1,\omega_2}\in{\mathcal P}(A)$).

II. The operator $A^{\omega_1,\omega_2}$ is called {singularly
  rank-one nonsymmetric \emph{parametrically}} perturbed with respect
to the operator $A$ if $(A^{1/2}\eta_2,A^{1/2}\eta_1)$ does not exist.
Moreover,
\begin{multline}\label{f5}
{\mathfrak D}(A^{\omega_1,\omega_2})=
\Big\{\psi=\varphi+b\eta_2 \ \vert \ \varphi\in{\mathfrak D}(A),  \\ \left.
b=b(\varphi)=\frac{(A\varphi,\eta_1)}{1+\tau+(A^{1/2}(A^2+1)^{-1/2}\eta_2,A^{1/2}(A^2+1)^{-1/2}\eta_1)}\right\}
\end{multline}
in the case $(A^{1/2}(A^2+1)^{-1/2}\eta_2,A^{1/2}(A^2+1)^{-1/2}\eta_1)\not=-\tau-1$,
where $\tau\in{\mathbb C}$ is a parameter, and
\begin{equation}\label{f6}
{\mathfrak D}(A^{\omega_1,\omega_2})={\mathfrak D}_{{\mathcal H}_1}\dot +\{c\eta_2\}, \quad
{\mathfrak D}_{{\mathcal H}_1}=\left\{\varphi\in{\mathfrak D}(A) \ \vert \
(A\varphi,\eta_1)=0 \right\}
\end{equation}
in the case $(A^{1/2}(A^2+1)^{-1/2}\eta_2,A^{1/2}(A^2+1)^{-1/2}\eta_1)=-\tau-1$,
(this fact we
denote by $A^{\omega_1,\omega_2}\in{\mathcal P}_{\tau}(A)$)

The action of the perturbed operator is given by the rule
\[
A^{\omega_1,\omega_2}\psi=A\varphi
\]
in each case.
\end{deff}

{\bf Remark 1}. If $\vert (A^{1/2}\eta_2,A^{1/2}\eta_1) \vert<\infty$ in the second part of
Definition \ref{d1}, then taking
$\tau=(A^{1/2}(A^2+1)^{-1/2}\eta_2,A^{1/2}(A^2+1)^{-1/2}\eta_1)$ we obtain the first part of
Definition \ref{d1}, that was considered also in \cite{VD}.

{\bf Remark 2}. If $(A^{1/2}\eta_2,A^{1/2}\eta_1)$  does not exists i.e., we have
the second part of Definition \ref{d1}, then still
$\vert (A^{1/2}(A^2+1)^{-1/2}\eta_2,A^{1/2}(A^2+1)^{-1/2}\eta_1) \vert<\infty$, since
$\omega_1, \omega_2\in{\mathcal H}_{-2}$.

{\bf Remark 3}.  The defined in Definition \ref{d1} operator
$A^{\omega_1,\omega_2}$ can be described in the following way.  A
linear closed operator $A^{\omega_1,\omega_2}\not=A$ densely defined
on ${\mathcal H}$ is nonsymmetric singularly perturbed with respect to
the operator $A$ if both sets,
\begin{align}
{\mathfrak D}&=\{f\in{\mathfrak D}(A)\cap{\mathfrak D}(A^{\omega_1,\omega_2}) \ \vert \ Af=\tilde Af\},\label{eq2} \\
{\mathfrak D}_*&=\{f\in{\mathfrak D}(A)\cap{\mathfrak D}((A^{\omega_1,\omega_2})^*) \ \vert \ Af=\tilde A^*f\}\label{eq3},
\end{align}
are dense in ${\mathcal H}$. In general $\tilde A\in{\mathcal P}_{\tau}(A)$.

It is clear that for each operator $A^{\omega_1,\omega_2}\in{\mathcal
  P}_{\tau}(A)$, there exist densely defined symmetric restrictions,
i.e., operators $\dot A:=A\restriction{\mathfrak D}$ and $\dot
A_*:=A\restriction{\mathfrak D}_*$ with nontrivial deficiency indices
\[
{\bf n}^{\pm}(\dot A)={\rm dim \ ker}(\dot A\mp z)^*\not=0, \quad  {\bf n}^{\pm}(\dot A_*)={\rm dim \ ker}(\dot
A_*\mp z)^*\not=0, \quad z\in\rho(A).
\]
(In this article we meet often the case where ${\bf n}^{\pm}(\dot
A)={\bf n}^{\pm}(\dot A_*)=1$.)

If ${\mathfrak D}={\mathfrak D}_*$ and $\tilde A=\tilde A^*$, then we
are in the case of the usual abstract definition of singularly
perturbed selfadjoint operators \cite{AKu1,Ko17} $\tilde A\in{\mathcal
  P}_s(A)$, that is, the definition given above generalizes the known
definition of a selfadjoint singular perturbation to the case of a
nonselfadjoint one.

The operator defined above, $A^{\omega_1,\omega_2}$, has the following general properties.
\begin{prop}\label{pp1}
For an arbitrary nonzero constant $a\in{\mathbb C}$ we have
$A^{a\omega_1,\omega_2}=A^{\omega_1,\bar a\omega_2}$.
\end{prop}
\begin{proof}
  From Definition \ref{d1} in both cases (\ref{f3}), (\ref{f4}) and
  (\ref{f5}), (\ref{f6}) it follows that ${\mathfrak
    D}(A^{a\omega_1,\omega_2})={\mathfrak D}(A^{\omega_1,\bar
    a\omega_2})$ and $A^{a\omega_1,\omega_2}\psi=A^{\omega_1,\bar
    a\omega_2}\psi=A\varphi$.
\end{proof}
\begin{prop}\label{pp2}
The adjoint operator $(A^{\omega_1,\omega_2})^*$ satisfies the identity
\[
(A^{\omega_1,\omega_2})^*=A^{\omega_2,\omega_1}.
\]
\end{prop}
\begin{proof}
For the proof we use  the second part of Definition \ref{d1} for
$A^{\omega_1,\omega_2}$ and $A^{\omega_2,\omega_1}$ and verify the identity
\begin{equation}\label{ff1}
(A^{\omega_1,\omega_2}f_1,f_2)=(f_1,A^{\omega_2,\omega_1}f_2),
\end{equation}
for $f_1\in{\mathfrak D}(A^{\omega_1,\omega_2})$ and
$f_2\in{\mathfrak D}(A^{\omega_2,\omega_1})$ of the form
$f_1=\varphi_1+b_1\eta_2$ and $f_2=\varphi_2+b_2\eta_1$, correspondingly, $\varphi_1,\varphi_2\in{\mathfrak D}(A)$.
The left-hand side  of (\ref{ff1}) is of the form
\begin{equation}\label{ff2}
(A^{\omega_1,\omega_2}f_1,f_2)=(A^{\omega_1,\omega_2}(\varphi_1+b_1\eta_2),(\varphi_2+b_2\eta_1))=
(A\varphi_1,\varphi_2)+\bar b_2(A\varphi_1,\eta_1),
\end{equation}
where
\[
b_1=b_1(\varphi_1)=\frac{(A\varphi_1,\eta_1)}{1+\tau+(A^{1/2}(A^2+1)^{-1/2}\eta_2,A^{1/2}(A^2+1)^{-1/2}\eta_1)}.
\]
The right-hand side  of (\ref{ff1}) is of the form
\begin{equation}\label{ff3}
(f_1,A^{\omega_2,\omega_1}f_2)=((\varphi_1+b_1\eta_2),A^{\omega_2,\omega_1}(\varphi_2+b_2\eta_1))=
(\varphi_1,A\varphi_2)+b_1(\eta_2,A\varphi_2),
\end{equation}
where
\[
b_2=b_2(\varphi_2)=\frac{(A\varphi_2,\eta_2)}{1+\tau+(A^{1/2}(A^2+1)^{-1/2}\eta_1,A^{1/2}(A^2+1)^{-1/2}\eta_2)}.
\]

The statement of the proposition follows from the obvious equality of last terms from (\ref{ff2})
and (\ref{ff3}),
\begin{multline*}
\frac{\overline{(A\varphi_2,\eta_2)}}{1+\tau+\overline{(A^{1/2}(A^2+1)^{-1/2}\eta_1,A^{1/2}(A^2+1)^{-1/2}\eta_2)}}(A\varphi_1,\eta_1)\\
=\frac{(A\varphi_1,\eta_1)}{1+\tau+(A^{1/2}(A^2+1)^{-1/2}\eta_2,A^{1/2}(A^2+1)^{-1/2}\eta_1)}(\eta_2,A\varphi_2).
\end{multline*}

The proof in case (\ref{f5}), i.e.,
$(A^{1/2}(A^2+1)^{-1/2}\eta_2,A^{1/2}(A^2+1)^{-1/2}\eta_1)\not=-\tau-1$
is completed.  The case (\ref{f6}), i.e.,
$(A^{1/2}(A^2+1)^{-1/2}\eta_2,A^{1/2}(A^2+1)^{-1/2}\eta_1)=-\tau-1$ is
also valid.

The cases (\ref{f3}), (\ref{f4}) are particular with respect to
(\ref{f5}), (\ref{f6}).
\end{proof}
\section{The description of a rank-one nonsymmetric singular perturbation by resolvents}
In this section we consider the perturbed operator with a parameter
$\alpha\in{\mathbb C}$ and denote it by $\tilde
A=A+\alpha\langle\cdot,\omega_1\rangle\omega_2$, where
$\omega_i\in{\mathcal H}_{-2}\setminus{\mathcal H}$ and
$\Vert\omega_i\Vert_{-1}=1$, $i=1,2$.  The set of such operators is
also denoted by ${\mathcal P}_{\tau}(A)$.  At the beginning, let us
briefly remark that if $\tilde A\in{\mathcal P}_{\tau}(A)$ then for
the adjoint operator, we have $\tilde A^*\in{\mathcal P}_{\tau}(A)$, which is also due to the
investigations in \cite{BB} and Proposition \ref{pp2}.
\begin{theorem}\label{t1}
  For the resolvents $R_z=(A-z)^{-1}$ and $\tilde R_z=(\tilde
  A-z)^{-1}$ of operators $A=A^*>1$ and $\tilde A\in{\mathcal
    P}_{\tau}(A)$ on a separable Hilbert space ${\mathcal H}$, a formula
   similar to  M.~Krein's formula holds for $z, \xi,
  \zeta\in\rho(A)\cap\rho(\tilde A)$,
\begin{equation}\label{f10}
\tilde R_z=R_z+b_z(\cdot,n_{\bar z})m_z,
\end{equation}
with
\begin{equation}\label{f11}
n_z=(A-\xi)(A-z)^{-1}n_{\xi}, \quad  m_z=(A-\zeta)(A-z)^{-1}m_{\zeta},
\end{equation}
where $n_z,m_z\in{\mathcal H}\setminus{\mathcal H}_{+2}$ and
\begin{equation}\label{f12}
b_z^{-1}-b_{\xi}^{-1}=(\xi-z)(m_{\xi},n_{\bar z}).
\end{equation}
The vectors $n_z, m_z$ and the number $b_z$ are connected with $\omega_1, \omega_2$ as follows:
\begin{equation}\label{f13}
n_z=R_z\omega_1, m_z=R_z\omega_2,  \quad -b_z^{-1}={\alpha}^{-1}+\tau+
\langle(A^2+1)^{-1}\omega_2,(1+\bar zA)R_{\bar z}\omega_1\rangle,
\end{equation}
where $\alpha\not=0$.
\end{theorem}
The case $\alpha=0$ can be also included in the considerations, since in case  $\alpha=0$,
we put $b_z\equiv 0$ and obtain  $\tilde R_z\equiv R_z$.
\begin{proof}
  From the expression $\tilde
  A=A+\alpha\langle\cdot,\omega_1\rangle\omega_2$ and for some
  $z\in\rho(A)\cap\rho(\tilde A)$ we have $\tilde
  A-z=A-z+\alpha\langle\cdot,\omega_1\rangle\omega_2$, and hence
\begin{equation}\label{f14}
(\tilde A-z)^{-1}=(A-z)^{-1}-\alpha\langle\cdot,(A-\bar z)^{-1}\omega_1\rangle(\tilde A-z)^{-1}\omega_2,
\end{equation}
where $(A-z)^{-1}$ and $(\tilde A-z)^{-1}$ are considered as operators
from ${\mathcal H}_{-2}$ into ${\mathcal H}$.  In the case of the
first part of Definition \ref{d1} we can continue our proof as in
\cite{VD}.  In the case of the second (general) part of Definition
\ref{d1} we need to consider the next expression following from
(\ref{f14}):
\begin{multline}\label{f15}
(\tilde A-z)^{-1}=(A-z)^{-1}-\alpha\{
\langle A(A^2+1)^{-1/2}\cdot,A(A-\bar z)^{-1}\omega_1\rangle \\
+\langle(A^2+1)^{-1}\cdot,(1+\bar zA)(A-\bar z)^{-1}\omega_1\rangle
\}(\tilde A-z)^{-1}\omega_2.
\end{multline}
In particular, for $\omega_2\in{\mathcal H}_{-2}$ we have
\[
(\tilde A-z)^{-1}\omega_2=(A-z)^{-1}\omega_2-\alpha\{
\tau +\langle(A^2+1)^{-1}\omega_2,(1+\bar zA)(A-\bar z)^{-1}\omega_1\rangle
\}(\tilde A-z)^{-1}\omega_2.
\]
where instead of  $\langle \omega_2,A(A^2+1)^{-1} \omega_1 \rangle$ we write $\tau$.
Substituting
\[
(\tilde A-z)^{-1}\omega_2=
\frac{\alpha}{1+\alpha\{\tau +\langle(A^2+1)^{-1}\omega_2,(1+\bar zA)(A-\bar z)^{-1}\omega_1\rangle\}}(A-z)^{-1}\omega_2.
\]
into (\ref{f15}) we get
\begin{multline}
(\tilde A-z)^{-1}=(A-z)^{-1}-
\frac{1}{\alpha^{-1}+\tau+
\langle(A^2+1)^{-1}\omega_2,(1+\bar zA)(A-\bar z)^{-1}\omega_1\rangle}\\
\times(\cdot,(A-\bar z)^{-1}\omega_1)(A-z)^{-1}\omega_2.
\end{multline}
where $\alpha\not=0$.
If we put
\begin{equation}\label{f16}
n_{z}=(A-z)^{-1}\omega_1, \quad  m_{z}=(A-z)^{-1}\omega_2,
\end{equation}
and
\begin{equation}\label{f17}
b_z^{-1}=-({\alpha^{-1}+\tau+\langle(A^2+1)^{-1}\omega_2,(1+\bar zA)(A-\bar z)^{-1}\omega_1\rangle}),
\end{equation}
then we obtain  (\ref{f13}).

Analogously, starting with $\tilde
A^*=A+\bar\alpha\langle\cdot,\omega_2\rangle\omega_1$, we also go to
(\ref{f16}) and (\ref{f17}) in the equivalent form
\[
\bar b_z^{-1}=-({\bar\alpha^{-1}+\bar \tau+\langle(A^2+1)^{-1}\omega_1,(1+zA)(A-z)^{-1}\omega_2\rangle}),
\]

Let us remark that if  $\omega_1,\omega_2\in{\mathcal H}_{-2}\setminus{\mathcal H}$ then
$n_{z}, m_{z}\in{\mathcal H}\setminus{\mathcal H}_{2}$ (see, for example \cite{Ko15}),
and, more precisely, if $\omega_1,\omega_2\in{\mathcal H}_{-2}\setminus{\mathcal H}_{-1}$, then
$n_{z}, m_{z}\in{\mathcal H}\setminus{\mathcal H}_{+1}$ and
if $\omega_1,\omega_2\in{\mathcal H}_{-1}\setminus{\mathcal H}$, then
$n_{z}, m_{z}\in{\mathcal H}_{+1}\setminus{\mathcal H}_{+2}$.

By using notations (\ref{f16}) in the form
$n_{z}=(A-z)^{-1}\omega_1$ and $n_{\xi}=(A-\xi)^{-1}\omega_1$ we
obtain $\omega_1=(A-z)n_{z}=(A-\xi)n_{\xi}$ and consequently the first
expression in (\ref{f11}),
\[
n_z=(A-\xi)(A-z)^{-1}n_{\xi}.
\]
This expression makes sense in ${\mathcal H}$ if we consider $A$ on ${\mathcal H}$ (but not
${\bf A}$).
Analogously we obtain the second expression in (\ref{f11}).
By a use of the Hilbert identity with (\ref{f11}), we obtain (\ref{f12}),
\begin{equation*}
\begin{aligned}
b_{z}^{-1}-&b_{\xi}^{-1}\\
&=-\langle(A^2+1)^{-1}\omega_2,(1+\bar zA)(A-\bar z)^{-1}\omega_1\rangle+
\langle(A^2+1)^{-1}\omega_2,(1+\bar \xi A)(A-\bar\xi)^{-1}\omega_1\rangle\\
&=\langle(A^2+1)^{-1}\omega_2,\left( (1+\bar zA)(A-\bar \xi)^{-1}-(1+\bar \xi A)(A-\bar z)^{-1} \right)\omega_1\rangle \\
&=(\xi-z)\langle (A-\xi)^{-1}\omega_2, (A-\bar z)^{-1}\omega_1\rangle \\
&=(\xi-z)(m_{\xi},n_{\bar z}).
\end{aligned}
\end{equation*}
This completes the proof. 
\end{proof}

Let us remark that  it is possible that  $b_z=\infty$, and it is so iff $z\in\sigma_p(\tilde A)$, but
(\ref{f10}) is also valid in such a case.

At the end of the article, we give an example that illustrates the
following corollary from Theorem \ref{t1}.
\begin{cor}\label{c1}
  If the operators $A=A^*$ and $\tilde A\in{\mathcal P}(A)$ have
  inverses on a separable Hilbert space ${\mathcal H}$, i.e.,
  $0\in\rho(A)\cap\rho(\tilde A)$, then (\ref{f10}), (\ref{f11}) and
  (\ref{f12}) have the form
\begin{equation}\label{f18}
\tilde A^{-1}=A^{-1}+b_0(\cdot,n_0)m_0,
\end{equation}
where
\begin{equation}\label{f19}
n_0=A^{-1}\omega_1, \quad m_0=A^{-1}\omega_2, \quad
-b_0^{-1}={\alpha}^{-1}+\langle\omega_2,A^{-1}\omega_1\rangle.
\end{equation}
\end{cor}
\begin{proof} 
  The proof follows from Theorem \ref{t1} in
  two ways. The first one is to take into account that
  $\vert\langle\omega_2,A^{-1}\omega_1\rangle\vert<\infty$ and repeat
  the proof of Theorem \ref{t1}. The second one is to substitute
  $z=0$ into (\ref{f10}), (\ref{f11}) and (\ref{f12}).
\end{proof}
\section{Spectral properties of rank one  nonsymmetric singular perturbations}
As a starting point, let us remark that the continuous spectrum
$\sigma_c(A)$ of the operator $A$ is unchanged by finite rank
perturbations, i.e., $\sigma_c(A)=\sigma_c(\tilde A)$, $\tilde
A\in{\mathcal P}_{\tau}(A)$.
\begin{theorem}\label{t2}
  Let the perturbed operator $\tilde A\in{\mathcal P}_{\tau}(A)$
  possess a new eigenvalue $\lambda\in{\mathbb C}$ comparing to $A$,
  i.e., there exist $\lambda\in\sigma_p(\tilde A)$,
  $\lambda\not\in\sigma_p(A)$. Then for the corresponding eigenvectors
  $\varphi$, $\psi$: $\tilde A\varphi=\lambda\varphi$ and $\tilde
  A^*\psi=\bar\lambda\psi$, the following relations hold true:
\begin{equation}\label{fx1+}
(\lambda-z) b_z(\varphi,n_{\bar z})=1, \quad \varphi=(A-z)(A-\lambda)^{-1}m_{z};
\end{equation}
\begin{equation}\label{fx11+}
(\bar\lambda-\bar z) \bar b_z(\psi,m_z)=1, \quad \psi=(A-\bar z)(A-\bar\lambda)^{-1}n_{\bar z}.
\end{equation}
\end{theorem}
\begin{proof} The proof follows from Theorem \ref{t1}.  Let $\tilde
  A\varphi=\lambda\varphi$, i.e.,
\[
\tilde R_z\varphi=R_z\varphi+b_z(\varphi,n_{\bar z})m_z=(\lambda-z)^{-1}\varphi.
\]
Hence, $ b_z(\varphi,n_{\bar z})m_z=[(\lambda-z)^{-1}-R_z]\varphi=
(\lambda-z)^{-1}(A-\lambda)(A-z)^{-1}\varphi, $
\begin{equation}\label{fx1}
(\lambda-z)b_z(\varphi,n_{\bar z})(A-z)(A-\lambda)^{-1}m_z=\varphi.
\end{equation}
Multiplying the last expression  by $n_{\bar z}$  we obtain that
\[
(\lambda-z)b_z(\varphi,n_{\bar z})\left((A-z)(A-\lambda)^{-1}m_z,n_{\bar z}\right)=(\varphi,n_{\bar z}),
\]
and, hence,
\begin{equation}\label{fx11}
(\lambda-z)b_z(\varphi,n_{\bar z})=1.
\end{equation}
We remark that  (\ref{fx1}) and (\ref{fx11}) gives  $\varphi=(A-z)(A-\lambda)^{-1}m_z$.
This proves (\ref{fx1+}).
Analogously, considering  $\tilde A^*\psi=\bar\lambda\psi$ we can prove (\ref{fx11+}).
\end{proof}
\begin{cor}
  Let us put $z=0$ in the case $\tilde A\in{\mathcal P}(A)$ under the
  conditions of Theorem \ref{t2}. Then (\ref{fx1+}) and (\ref{fx11+})
  have the form
\[
\lambda b_0(\varphi,n_0)=1, \quad \varphi=A(A-\lambda)^{-1}m_0; \quad
\bar\lambda \bar b_0(\psi,m_0)=1, \quad \psi=A(A-\bar\lambda)^{-1}n_0.
\]
\end{cor}
\begin{prop}\label{p2}
Let the perturbed operator $\tilde A\in{\mathcal P}(A)$ possess a new eigenvalue
$\lambda\in{\mathbb C}$ in comparison with $A$ and let the eigenvectors be
$\varphi$ and $\psi$, i.e., $\tilde A\varphi=\lambda\varphi$ and $\tilde A^*\psi=\bar\lambda\psi$
then the relations (\ref{fx1}) and (\ref{fx11}) , in terms of $\omega_1, \omega_2$, have a form
\begin{equation}\label{fnd}
\alpha\langle(A-\lambda)^{-1}\omega_2,\omega_1\rangle=-1, \quad \varphi=(A-\lambda)^{-1}\omega_2;
\end{equation}
\begin{equation}\label{fnd+}
\bar\alpha\langle(A-\bar\lambda)^{-1}\omega_1,\omega_2\rangle=-1, \quad \psi=(A-\bar\lambda)^{-1}\omega_1.
\end{equation}
\end{prop}
\begin{proof} Taking $n_z$ and $m_z$ as in (\ref{f13}), we obtain
  (\ref{fnd}) and (\ref{fnd+}).  The second way to prove is the
  following.  Instead of (\ref{f10}) we take the representation
  $\tilde A=A+\alpha \langle \cdot,\omega_1\rangle\omega_2$ and
  conduct similar conversions with the corresponding calculations.
\end{proof}
\section{The inverse spectral problem for a rank one nonsymmetric
  singular perturbation}
If we regard the formulation of the Theorem \ref{t2} as a direct spectral
problem, then we can consider the following Theorem as a
corresponding inverse problem.
\begin{theorem}\label{t3}
For a given positive selfadjoint operator $A=A^*$ on a separable
Hilbert space ${\mathcal H}$ and $\lambda\in{\mathbb C}$ and
vectors $\varphi,\psi\in{\mathcal H}\setminus{\mathcal H}_{+1}$
($\varphi,\psi\in{\mathcal H}_{+1}\setminus{\mathcal H}_{+2}$),
there exist a unique
$\tilde A\in{\mathcal P}_{\tau}(A)$
($\tilde A\in{\mathcal P}(A)$) such that
$\tilde A\varphi=\lambda\varphi$ and
$\tilde A^*\psi=\bar\lambda\psi$.
Moreover,
the operator $\tilde A$ is defined by (\ref{f10}) as follows:
\begin{equation}\label{f88}
\tilde R_z=R_z+b_z(\cdot,n_{\bar z})m_z,
\end{equation}
with
\begin{equation}\label{f23}
m_{z}=(A-\lambda)(A-z)^{-1}\varphi, \quad n_{\bar z}=(A-\bar\lambda)(A-\bar z)^{-1}\psi
\end{equation}
and
\begin{equation}\label{f24}
b_z^{-1}=(\lambda-z)(\varphi,n_{\bar z}), \quad  \left( \bar b_z^{-1}=(\bar\lambda-\bar z)(\psi,m_{z}) \right).
\end{equation}
\end{theorem}

In general, the proof of Theorem \ref{t3} does not differ from the
proof of the similar Theorem in \cite{VD} and hence we give  a
sketch of the proof there.

\begin{proof}
  Let us remark that if $\varphi,\psi\in{\mathcal H}\setminus{\mathcal
    H}_{+1}$, then there exists $\tilde A\in{\mathcal P}_{\tau}(A)$
  (possibly $\tilde A\in{\mathcal P}(A)$) and if
  $\varphi,\psi\in{\mathcal H}_{+1}\setminus{\mathcal H}_{+2}$, then
  there exists exactly $\tilde A\in{\mathcal P}(A)$.

  The proof needs a supplementary proposition that has a general
  character.  One of them is known from \cite{D}.
\begin{prop}\label{p3}
  Let there be given a (half-bounded) positive selfadjoint operator
  $A$ with  domain ${\mathfrak D}(A)$ in a separable Hilbert
  space ${\mathcal H}$.  Then for an arbitrary vector $\eta\in
  {\mathcal H}\setminus {\mathfrak D}(A)$ and an arbitrary number
  $z\in{\mathbb C}$, ${\rm Im}(z)\not=0$, there exists a restriction
  $\dot A$ of the operator $A$ such that $\eta =n_z$ is its defect
  vector namely $(\dot A-z)^{*}n_z=0$.
\end{prop}
This proposition is used for $\dot A$ and $\dot A_*$.

Next we use the result which, in some sense, is inverse to the one
pointed out in Theorem \ref{t1}, --- that is a formula similar to
M.~Krein's formula but from the perturbation point of view, i.e., the
perturbation of the resolvent of the selfadjoint operator by a
one-dimensional skew projection.
\begin{prop}\label{p4}
  Let there be given a positive selfadjoint operator $A$ on a
  separable Hilbert space ${\mathcal H}$.  The operator-valued
  function
\begin{equation}\label{f25}
\tilde R_z:=(A-z)^{-1}+b_z(\cdot ,n_{\bar z})m_z
\end{equation}
\begin{itemize}
\item[{\bf 1)}] is a resolvent of a closed operator, if $n_z$  and $b_z$ satisfy conditions
(\ref{f11}) and (\ref{f12}) and
$n_{\bar z}$, ($m_{\bar z}$) is not an eigenvector of $A-z$, ($A-\bar z$);
\item[{\bf 2)}] is the resolvent of a rank one singularly perturbed
  operator iff, additionally to the previous assumption, the
  inclusions $n_z,m_z\in{\mathcal H}\setminus{\mathcal H}_{+2}$ hold
  true.
\end{itemize}
\end{prop}

Let us give also a sketch of the proof.  The part {\bf 1}) is verified
due to Theorem 7.7.1 \cite{Ka} as follows: $\tilde R_z$ is the
resolvent of a closed operator iff
\begin{itemize}
\item[{\bf a)}] $\tilde R_z$ satisfies the Hilbert identity
$
\tilde R_z-\tilde R_{\xi}=(z-\xi)\tilde R_z\tilde R_{\xi}, \ {\rm Im} z,\xi\not=0;
$
\item[{\bf b)}] $\tilde R_z$  has the trivial kernel,
$
\ker (\tilde R_z)=\{0\}, \ {\rm Im} z\not=0.
$
\end{itemize}

The part {\bf a)} is verified by substituting (\ref{f25}) into the
Hilbert identity for the resolvent.  The condition {\bf b)} is
verified by directly checking (\ref{f25}) for the vectors $f\perp
n_{\bar z}$ and $n_{\bar z}$.  The condition {\bf 2)} follows from the
fact that $n_z,m_z\in{\mathcal H}\setminus{\mathfrak D}(A)$.

We now return to the sketch of the proof of Theorem \ref{t3}.  Taking
$n_z,m_z$ in the form (\ref{f23}) and (\ref{f24}) we will check
identity (\ref{f12}).  Since the vectors $n_z$ and $m_z$ belong to
${\mathcal H}\setminus{\mathcal H}_{+2}$, by Proposition \ref{p4},
i.e., by part {\bf 2}), the operator $\tilde A$ is singularly
perturbed with respect of $A$.  The identity $\tilde A\varphi
=\lambda\varphi$ is checked by direct calculation.  The uniqueness is
proved by contradiction.
\end{proof}
\begin{prop}
  For a given selfadjoint operator $A$ on a separable Hilbert space
  ${\mathcal H}$, and $\lambda\in{\mathbb C}$, and vectors
  $\varphi,\psi\in{\mathcal H}_{+1}\setminus{\mathcal H}_{+2}$, there
  exist a unique $\tilde A\in{\mathcal P}(A)$ such that $\tilde
  A\varphi=\lambda\varphi$, $\tilde A^*\psi=\bar\lambda\psi$.
  Moreover, the operator $\tilde A$ is defined by the expression $
  \tilde A=A+\alpha\langle\cdot,\omega_1\rangle\omega_2, $ where $
  \omega_1=(A-\bar\lambda)\psi, \quad \omega_2=(A-\lambda)\varphi, $
  and $ \alpha^{-1}=-\langle(A-\lambda)\omega_2,\omega_1\rangle$, or $
  \alpha^{-1}=-\langle(A-\bar\lambda)\omega_1,\omega_2\rangle.  $
\end{prop}
\begin{proof} The corresponding proof is an implication of Theorems
  \ref{t1} and \ref{t3}.
\end{proof}
\section{A dual pair of eigenvalues}
Since $\tilde A\in{\mathcal P}_{\tau}(A)$ is a non-self-adjoint
operator, the definition of a dual pair of eigenvalues is different from
\cite{ADK}.
\begin{deff}
  A couple of numbers $\lambda,\mu\in{\mathbb C}$ is called a dual
  pair of eigenvalues of a singularly perturbed operator $\tilde
  A\in{\mathcal P}_{\tau}(A)$ iff
\begin{align} \label{f*}
&\tilde A \varphi_{\lambda}=\lambda\varphi_{\lambda},  \quad \tilde A\varphi_{\mu}=\mu\varphi_{\mu}, \\
&\tilde A^* \psi_{\bar\lambda}=\bar\lambda\psi_{\bar\lambda}, \quad  \tilde A^* \psi_{\bar\mu}=\bar\mu\psi_{\bar\mu}, \label{f*2} \\
&(\bar\lambda-\bar\mu)((A-\mu)^{-1}\varphi_{\lambda}, \psi_{\bar\lambda})=(\varphi_{\lambda},\psi_{\bar \lambda}).\label{f*3}
\end{align}
\end{deff}
The next theorem describes a method how to construct an operator with
a dual pair.

\begin{theorem}\label{t7}
  Let $A=A^*\geq c$ be a semibounded selfadjoint operator defined on
  ${\mathfrak D}(A)$ in a separable Hilbert space ${\mathcal
    H}$. For an arbitrary $\mu\in \rho (A)$ and vectors
  $\varphi_{\lambda}, \psi_{\bar\lambda}\in {\mathcal H}\setminus
  {\mathcal H}_{+2}$ there exists a unique nonsymmetric singularly
  perturbed operator $\tilde A \in {\mathcal P}_{\tau}(A)$ such that
  $(\mu, \lambda)$ is a dual pair, where
  $\bar\lambda:=\bar\mu+\frac{(\varphi_{\lambda},\psi_{\bar\lambda})}{((A-\mu)^{-1}\varphi_{\lambda},\psi_{\bar\lambda})}$
  is an eigenvalue with the eigenvector
  $\varphi_{\mu}=(A-\lambda)(A-\mu)^{-1}\varphi_{\lambda}$.  The
  adjoint operator $\tilde A^*\in{\mathcal P}_{\tau}(A)$ has also
  eigenvectors $\psi_{\bar\lambda}$ and
  $\psi_{\bar\mu}=(A-\bar\lambda)(A-\bar{\mu})\psi_{\bar{\lambda}}$.
\end{theorem}

{\bf Remark 4}.
If $\tilde A\in{\mathcal P}(A)$, then in the form
$\tilde A=A+\alpha\langle\cdot,\omega_1\rangle\omega_2$, we can calculate  the
coupling constant $\alpha=-\frac{1}{(\varphi_{\lambda}, \omega_1)}$ (or
$\bar\alpha=-\frac{1}{(\psi_{\bar\lambda},\omega_2)}$) and the corresponding vectors
\begin{equation}\label{eq5}
\omega_2=(A-\mu)\varphi_{\lambda}-\frac{(\psi_{\bar{\lambda}}, \varphi_{\lambda})}{((A-\mu)\psi_{\bar{\lambda}},
\varphi_{\lambda})}\varphi_{\lambda}, \quad \omega_1=(A-\bar\mu)\psi_{\bar{\lambda}}-\frac{(\varphi_{\lambda},
\psi_{\bar\lambda})}{((A-\mu)\varphi_{\lambda}, \psi_{\bar\lambda})}\psi_{\bar{\lambda}}.
\end{equation}

\begin{proof} 
  The proof is given by a direct verification of only (\ref{f*}) and
  (\ref{f*2}, since (\ref{f*3} is the condition of
  Theorem. \end{proof}

For a real dual pair $\lambda,\mu\in{\mathbb R}$ we have the
following corollary that follows from Theorem \ref{t7}.
\begin{cor}\label{c8}
  Let $A=A^*\geq 0$ be a positive selfadjoint operator defined on
  ${\mathfrak D}(A)$ in a separable Hilbert space ${\mathcal H}$, so
  that $\sigma(A)=\sigma_c(A)=[0,\infty)$. For an arbitrary number
  $\mu<0$ and vectors $\varphi_{\lambda},\psi_{\bar\lambda}\in
  {\mathcal H}_{+1}\setminus {\mathcal H}_{+2}$ there exists a unique
  nonsymmetric singular perturbation of rank one, $\tilde
  A\in{\mathcal P}(A)$, such that it has the dual pair
  $(\mu,\lambda)$, where
  $\lambda=\mu+\frac{(\varphi_{\lambda},\psi{\lambda})}{((A-\mu)^{-1}\varphi_{\lambda},\psi_{\lambda})}$,
  as its eigenvalues with the eigenvectors $\varphi_{\lambda}$ and
  $\varphi_{\mu}=(A-\lambda)(A-\mu)^{-1}\varphi_{\lambda}$.  Moreover,
  the operator $\tilde A^* \in {\mathcal P}(A)$ has the same
  eigenvalues but with different eigenvectors, $\psi_{\lambda}$ and
  $\psi_{\mu}=(A-\lambda)(A-\mu)\psi_{\lambda}$.
\end{cor}
\section{Approximations properties of perturbed operators}
Different approximations of singularly perturbed selfadjoint operators
are presented in \cite{AKKN}. We have some generalization to the case
of a nonsymmetric rank one perturbation.
\begin{theorem}
  Let $A$ be a semibounded self-adjoint operator defined on
  ${\mathfrak D}(A)$ in a separable Hilbert space ${\mathcal H}$, and
  two vectors $\omega_1,\omega_2\in{\mathcal H}_{-2}$ such that
  $\langle\omega_2,(A-z)^{-1}\omega_1\rangle$ does not exist.  Then
  there exist two sequences $\omega_{1,n}, \omega_{2,n}\in{\mathcal
    H}_{-2}$ converging to $\omega_i$, $i=1,2$, correspondingly, such
  that the sequence of operators $\tilde
  A_n=A+\alpha\langle\cdot,\omega_{1,n}\rangle\omega_{2,n}\in{\mathcal
    P}(A)$ converge to the operator $\tilde
  A=A+\alpha\langle\cdot,\omega_{1}\rangle\omega_{2}\in{\mathcal
    P}_{\tau}(A)$ in the norm resolvent sense if
\[
\lim \limits_{n\to\infty}\langle\omega_{2,n},A(A^2+1)^{-1}\omega_{1,n}\rangle=\tau.
\]
\end{theorem}
\begin{proof}
  Without loss of generality let us assume that $A=A^*\geq 0$.
  Resolvents of the corresponding operators have the following forms:
\begin{multline}
(\tilde A_n-z)^{-1}=(A-z)^{-1} \\ -\frac{1}{\alpha^{-1}+\langle\omega_{2,n},(A-\bar z)^{-1}\omega_{1,n}\rangle}
\langle\cdot,(A-\bar z)^{-1}\omega_{1,n}\rangle(A-z)^{-1}\omega_{2,n};
\end{multline}
\begin{multline}
(\tilde A-z)^{-1}=(A-z)^{-1} \\ -\frac{1}{\alpha^{-1}+\langle\omega_{2},(1+\bar zA)(A-\bar z)^{-1}(A^2+1)^{-1}\omega_{1}\rangle}
\langle\cdot,(A-\bar z)^{-1}\omega_{1}\rangle(A-z)^{-1}\omega_{2}.
\end{multline}
The difference of the resolvents has the form
\begin{equation*}
\begin{split}
(\tilde A_n-&z)^{-1}-(\tilde A-z)^{-1}= \\
&-\frac{1}{\alpha^{-1}+\langle\omega_{2,n},(A-\bar z)^{-1}\omega_{1,n}\rangle}
\langle\cdot,(A-\bar z)^{-1}\omega_{1,n}\rangle(A-z)^{-1}\omega_{2,n}
\\
&+\frac{1}{\alpha^{-1}+\langle\omega_{2},(1+\bar zA)(A-\bar z)^{-1}(A^2+1)^{-1}\omega_{1}\rangle}
\langle\cdot,(A-\bar z)^{-1}\omega_{1}\rangle(A-z)^{-1}\omega_{2}
\\
&=\left[-\frac{1}{\alpha^{-1}+\langle\omega_{2,n},(A-\bar z)^{-1}\omega_{1,n}\rangle}\right.
\\
&+\left.\frac{1}{\alpha^{-1}+\langle\omega_{2},(1+\bar zA)(A-\bar z)^{-1}(A^2+1)^{-1}\omega_{1}\rangle}
\right]
\\
&\times \langle\cdot,(A-\bar z)^{-1}\omega_{1,n}\rangle(A-z)^{-1}\omega_{2,n}+
\\
&+\frac{1}{\alpha^{-1}+\langle\omega_{2},(1+\bar zA)(A-\bar z)^{-1}(A^2+1)^{-1}\omega_{1}\rangle}
\\
&\times \langle\cdot,\left((A-\bar z)^{-1}\omega_{1,n}-(A-\bar z)^{-1}\omega_{1}\right)\rangle(A-z)^{-1}\omega_{2}
\\
&+\frac{1}{\alpha^{-1}+\langle\omega_{2},(1+\bar zA)(A-\bar z)^{-1}(A^2+1)^{-1}\omega_{1}\rangle}
\end{split}
\end{equation*}
\begin{equation*}
\times \langle\cdot,(A-\bar z)^{-1}\omega_{1}\rangle\left( (A-z)^{-1}\omega_{2,n}-(A-z)^{-1}\omega_{2}\right).\qquad\qquad\qquad\qquad
\end{equation*}

To prove that the resolvents of $\tilde A_n$ converge to the resolvent $\tilde A$
with respect to the operator norm it is enough to show that
\begin{equation}\label{fn1}
\Vert (A-z)^{-1}\omega_{i,n}-(A-z)^{-1}\omega_{i}\Vert\longrightarrow 0, \quad i=1,2,
\end{equation}
and
\begin{multline}\label{fn3}
\langle\omega_{2,n},(A-\bar z)^{-1}\omega_{1,n}\rangle
\longrightarrow\langle\omega_{2,n},A(A^2+1)^{-1}\omega_{1,n}\rangle \\ +
\tau+\langle\omega_{2,n},(1+\bar zA)(A-\bar z)^{-1}(A^2+1)^{-1}\omega_{1,n}\rangle,
\end{multline}
as $n\longrightarrow \infty$.
Indeed, if $\omega_{i,n}\longrightarrow\omega_i$, in ${\mathcal H}_{-2}$, $i=1,2$, then
(\ref{fn1}) is valid.

Hence to prove the theorem it is sufficient to show that there exist
sequences of vectors $\omega_{i,n}$, $i=1,2$, with the properties
\[
\omega_{i,n}\longrightarrow \omega_1 \ \text{in} \ {\mathcal H}_{-2}, \ \text{and} \
\lim \limits_{n\to\infty}\langle\omega_{2,n},A(A^2+1)^{-1}\omega_{1,n}\rangle=\tau.
\]
Since the case $\omega_{i,n}\in{\mathcal H}_{-1}$ would be trivial, we
consider the general case $\omega_{i,n}\in{\mathcal H}_{-2}$.  Let us
introduce the real sequence $ a_n=\langle
E_{[0,n]}\omega_2,A(A^2+1)^{-1}\omega_1\rangle, $ where $E$ is the
spectral measure of $A$.  Since $\omega_i\not\in{\mathcal H}_{-1}$,
$i=1,2$, there exists an interval $[c_n,d_n]$ inside $[0,n]$, so that
$ b_n=\langle E_{[c_n,d_n]}\omega_2,A(A^2+1)^{-1}\omega_1\rangle $
could be such that $\vert b_n\vert>\vert\tau-a_n\vert$.  We choose the
sequence $
\omega_{i,n}=E_{[0,n]}\omega_i+\varepsilon_{i,n}E_{[c_n,d_n]}\omega_i,
\ i=1,2.  $ where $\varepsilon_{i,n}$ are taken to be
\[
\varepsilon_{1,n}=\frac{\sqrt{\vert\tau-a_n\vert}}{\sqrt{b_n}}, \quad
\varepsilon_{2,n}=({\rm sgn}(\tau-a_n){\rm sgn} b_n)\varepsilon_{1,n}.
\]
It is obvious that $\vert\varepsilon_{i,n}\vert\leq 1$ and
$[
\langle\omega_{2,n},A(A^2-1)^{-1}\omega_{1,n}\rangle=a_n+\varepsilon_{1,n}\varepsilon_{2,n}b_n=\tau.
$
And we also have that
$
\Vert\omega_{i,n}-\omega_i\Vert_{-2}\leq\Vert E_{[0,\infty]}\omega_i\Vert_{-2}\longrightarrow 0,
$
as $n\longrightarrow\infty$.
Hence $\tilde A_n$ converges to $\tilde A$ in the norm resolvent sense.
\end{proof}
\section{Wave operators}
Assuming that $i\not\in\sigma(\tilde A)$
we define the wave operators for $A$ and $\tilde A\in{\mathcal P}_{\tau}(A)$ in a usual way
\cite{AKu1},
\begin{equation}\label{W+-}
W_{\pm}(A,\tilde A)=s-\lim\limits_{t\rightarrow\pm\infty}U_t
=s-\lim\limits_{t\rightarrow\pm\infty}e^{i\tilde At}e^{-iAt}P^{ac},
\end{equation}
where
\begin{equation*}
P^{ac}f=\int_{-\infty}^{\infty}\hat f^{ac}(\lambda)\,dE^{ac}(\lambda)(A+i)^{-1}\omega_1,
\quad f^{ac}\in{\mathcal H}^{ac},
\end{equation*}
$P^{ac}$ and $\tilde P^{ac}$ denote the spectral projections onto the
absolutely continuous parts $\sigma_{ac}(A)$ and $\sigma_{ac}(\tilde
A)$ of the spectrum of the operators $A$ and $\tilde A$; ${\mathcal
  H}^{ac}$ and $\tilde {\mathcal H}^{ac}$ are the corresponding
subspaces,
\[
{\mathcal H}^{ac}=P^{ac}{\mathcal H}, \
\tilde{\mathcal H}^{ac}=\tilde P^{ac}{\mathcal H};
\]
\[
f=\int_{-\infty}^{\infty}\hat f(\lambda)\,dE(\lambda)g=
\int_{-\infty}^{\infty}\hat {\tilde f}(\lambda)\,d\tilde E(\lambda)\tilde g,
\quad g=(A+i)^{-1}\omega_1, \ \tilde g=(\tilde A+i)^{-1}\omega_2.
\]

This definition of  wave operators is correct due to Theorem
1.5 from \cite{Ka}, namely, $\tilde A$ is the spectral type operator
and what is more, with this connection $\tilde A$ is the scalar type
operator $\sigma_{ac}(A)=\sigma_{ac}(\tilde A)$, and there exist
wave operators for the couple $A$ and $\tilde A$.
\begin{theorem}
  Let, on a separable Hilbert space ${\mathcal H}$, there be given a
  self-adjoint operator $A$ and its nonsymmetric singular perturbation
  of the form
\[
\tilde A=A+\alpha\langle\cdot,\omega_1\rangle\omega_2, \quad
\omega_1, \omega_2\in{\mathcal H}_{-2}\setminus{\mathcal H}, \ \alpha\in{\mathbb C}.
\]

Then there exist the defined in  (\ref{W+-}) wave operators $W_{\pm}$
 in the form:
\begin{equation}\label{W+-f}
W_{\pm}f=\int_{-\infty}^{\infty}(1+\tau+\alpha F(\lambda\pm i0))\hat f^{ac}(\lambda)\,
d\tilde E(\lambda)(\tilde A+i)^{-1}\omega_2,
\end{equation}
where
\[
P^{ac}f=\int_{-\infty}^{\infty}\hat f^{ac}(\lambda)\,dE(\lambda)(A+i)^{-1}\omega_1, \quad
F(z)=\langle(A-z)^{-1}\omega_2,(1+\bar zA)(A^2+I)^{-1}  \omega_1\rangle.
\]
\end{theorem}

\begin{proof} 
  The proof can be carried out by a direct calculation as in
  \cite{AKu1} or by substituting the corresponding expressions in the
  formulas obtained in \cite{VD} and taking into account that $\tilde
  A\in{\mathcal P}_{\tau}(A)$ instead of $\tilde A\in{\mathcal P}(A)$.
\end{proof}

The adjoint operator has the form
\[
 W_+^*g= \int_{-\infty}^{\infty}\frac{1}{1+\tau+\alpha F(\lambda+i0)}\hat {\tilde g}^{ac}(\lambda)\,dE(A+i)^{-1}\omega_1,
\]
where
\[
 \tilde P^{ac}g= \int_{-\infty}^{\infty}\hat {\tilde g}^{ac}\,d\tilde E(\lambda)(\tilde A+i)^{-1}\omega_2,
\quad \tilde g=(\tilde A+i)^{-1}\omega_2=\frac{1}{1+\tau+F(-i)}g.
\]
Then $S(\tilde A,A)=W_+^*W_-$ is a scattering operator and
\begin{equation}\label{S AA}
 S(\tilde A,A,\lambda,\tau)=\frac {1+\tau+\alpha F(\lambda-i0)}{1+\tau+\alpha F(\lambda+i0)}.
\end{equation}
is a scattering matrix consisting of one (complex) number.
\section{Examples}
{\bf Example 1}. Let us illustrate Theorems \ref{t1}, \ref{t2}, and
\ref{t3}.  Let ${\mathcal H}=L_2([2,\infty),dx)=L_2$ and $A$ be the
operator of multiplication by the independent variable $x^2$, namely
\[
Af(x)=x^2f(x), \quad {\mathfrak D}(A)=\{f(x)\in L_2 \ \vert \ x^2f(x)\in L_2\}.
\]
It is obvious that the operator $A\geq{2}$ and $A$ has purely
absolutely continuous spectrum, i.e.,
$\sigma(A)=\sigma_{ac}(A)=[{2},\infty)$.

Let us put ${\mathcal H}_1=L_2([{2},\infty),x^2dx)$ and ${\mathcal
  H}_{-1}=L_2([{2},\infty),x^{-2}dx)$ be the dual spaces.  In such a
case, ${\mathfrak D}(A)={\mathcal H}_2=L_2([{2},\infty),x^4dx)$, and
${\mathcal H}_{-2}=L_2([{2},\infty),x^{-4}dx)$.

Let us take $\omega_1=\frac{1}{{x-1}}$ and $\omega_2=\frac{1}{{x+1}}$.
It is obvious that $\omega_1,\omega_2\in{\mathcal
  H}_{-1}\setminus{\mathcal H}$.  So we can illustrate the operator
$\tilde A\in{\mathcal P}(A)$ in the form (\ref{f10}).

If we suppose additionally that $\tilde A$ possess a new point of
spectrum $\lambda=0$, i.e., $0\in\sigma_p(\tilde A)$, then due to 
Proposition \ref{p2} we calculate $\alpha=-\frac{2}{1-\ln 3}$ by the
formula (\ref{fnd}), since
\[
\langle A^{-1}\omega_2,\omega_1\rangle=
\int_{2}^{\infty}\frac{dx}{x^2(x^2-1)}=\frac{1-\ln 3}{2}.
\]

Hence,
\[
\tilde Af(x)=xf(x)-\frac{2}{1-\ln 3}\frac{1}{{x+1}}\int_{2}^{\infty}\frac{f(x)}{{x-1}}\,dx.
\]
Thus
\[
\varphi=\frac{1}{x^2(x+1)}, \quad \psi=\frac{1}{x^2(x-1)}.
\]

To illustrate Theorem \ref{t1}, formulas (\ref{f10}) and Theorem
\ref{t3}, formulas (\ref{f23}), (\ref{f24}) (for the case where
$\lambda=0$) we must put
\[
n_z=\frac{1}{(x^2-z)(x-1)}, \quad
m_z=\frac{1}{(x^2-z)(x+1)}, \quad
b_z^{-1}=\frac{2}{1-\ln 3}-\int_{{2}}^{\infty}\frac{dx}{(x^2-z)(x^2-1)}.
\]
It is obvious that $\tilde A\in{\mathcal P}(A)$.

{\bf Example 2}. This example illustrates Theorem \ref{t7}.  Let
${\mathcal H}=L_2({\mathbb R}^1,dx)$ and $A$ be a Laplace operator,
namely $Af(x)=-f''(x)$, ${\mathfrak D}(A)=W_2^2({\mathbb R}^1)$ is the
Sobolev space. The operator $A\geq 0$ has  purely absolutely
continuous spectrum, namely $\sigma(A)=\sigma_c(A)=[0,\infty)$.

We put $\varphi_{\lambda}=e^{-\vert x-1\vert}$ and
$\psi_{\lambda}=e^{-\vert x+1\vert}$, $\mu=-1$ (we consider the case
$\lambda, \mu\in{\mathbb R}$ and $\lambda,\mu\in\rho(A)$).  To
calculate $\lambda$, we need
\[
(\varphi_{\lambda},\psi_{\lambda})=\int_{{\mathbb R}} e^{-\vert x-1\vert}e^{-\vert x+1\vert}\,dx=3e^{-2}.
\]
By using \cite{AGHKH}, we calculate
\begin{equation}\label{fl}
(A+1)^{-1}\varphi_{\lambda}=\int_{{\mathbb R}}\frac{1}{2} e^{-\vert x-\tau\vert}e^{-\vert \tau-1\vert}\,d\tau=
\left\{
\begin{array}{l}
\frac{x}{2}e^{-x+1},  \ x>1, \\ \frac{2-x}{2}e^{x-1}, \ x<1,
\end{array}
\right.
\end{equation}
and
$
((A+1)^{-1}\varphi_{\lambda},\psi_{\lambda})=\frac{13}{4}e^{-2}.
$
Also, $\lambda=-\frac{1}{13}<0$, $\alpha=-\frac{4}{13}e^2$. And from (\ref{eq5}) we have that
\[
\omega_1=\delta_{-1}(x)-\frac{12}{13}e^{-\vert x+1\vert}, \quad \omega_2=\delta_{+1}(x)-\frac{12}{13}e^{-\vert x-1\vert}.
\]
The operator $\tilde A=A+\alpha\langle\cdot,\omega_1\rangle\omega_2\in{\mathcal P}(A)$ is such that
$\tilde A\varphi_{\lambda}=\lambda\varphi_{\lambda}$,
$\tilde A^*\psi_{\lambda}=\lambda\psi_{\lambda}$,
$\tilde A\varphi_{\mu}=\mu\varphi_{\mu}$,
$\tilde A^*\psi_{\mu}=\mu\psi_{\mu}$, and
$\varphi_{\mu}=(A-\lambda)(A-\mu)^{-1}\varphi_{\lambda}=e^{-\vert x+1\vert}$,
$\psi_{\mu}=e^{-\vert x-1\vert}$.

{\bf Example 3}. We illustrate once more Theorem \ref{t1}.  Let
${\mathcal H}=L_2({\mathbb R}^3)$ and $A$ play the role of the Laplace
operator, namely $Af(x)=-\Delta f(x)$, ${\mathfrak
  D}(A)=W_2^2({\mathbb R}^3)$ is the Sobolev space.  The operator is
positive $A\geq 0$ and has purely absolutely continuous spectrum,
i.e., $\sigma(A)=\sigma_c(A)=[0,\infty)$.  Let us consider the
expression (\ref{f4}) that describes the $\delta$-interaction with
retardation, i.e., the formal expression $\tilde
A=-\Delta+\alpha\langle\cdot,\delta_0\rangle\delta_1$, where
$\delta_0$ is a $\delta$-function at the point $0=(0,0,0)\in{\mathbb
  R}^3$ and $\delta_{1}$ is a $\delta$-function at the point ${\bf
  1}=(1,0,0)\in{\mathbb R}^3$.

Using \cite{AGHKH} we can write the resolvent of such operators at the
regular point $i\in{\mathbb C}$ (corresponding to one parameter
family) i.e its integral kernel,
\[
(-\tilde\Delta-i)^{-1}(x,p)=
\frac{e^{-\vert x-p\vert}}{4\pi\vert x-p\vert}+\frac{1}{\alpha^{-1}+(4\pi e)^{-1}}
\frac{e^{-(\vert x\vert+\vert {\bf 1}-p\vert)}}{(4\pi)^2\vert x\vert \vert {\bf 1}-p\vert}
\]
since
$\vert\langle\delta_1,(A-\bar z)^{-1}\delta_0\rangle\vert=\frac{1}{4\pi e}$,
$z=\sqrt{i}$, we take ($\Im \sqrt{i}>0$).
It is obvious that $-\tilde \Delta\in{\mathcal P}(-\Delta)$.

{\bf Example 4}. Let us again illustrate Theorem \ref{t7}.  Let
${\mathcal H}=L_2([1,\infty),dx)$ and $A$ be the multiplication
operator on $x^2$, namely $Af(x)=x^2f(x)$, $f\in{\mathfrak D}(A)$,
where ${\mathfrak D}(A):=\{ f(x)\in L_2 \ | \ x^2f(x)\in L_2 \}$.  It
is obvious that $A\geq 1$ and ${\mathfrak D}(A)=\sigma (A)=[1,\infty
)$.  We put $0=\mu\notin\sigma (A)$ and
$\varphi=\varphi_{\lambda}=x^{-1\frac{1}{3}}$,
$\psi=\psi_{\lambda}=x^{-1\frac{2}{3}}$, $\varphi,\psi\in{\mathcal H}$
but $\varphi,\psi\not\in{\mathcal H}_{+1}$.  In particular ${\mathcal
  H}_{+2}=L_2([1,\infty ),x^4dx)$. Then
\[
(\varphi,\psi)=\int_{1}^{\infty}\frac{dx}{x^3}=\frac{1}{2}, \quad
((A-\mu)^{-1}\varphi,\psi)=\int_{1}^{\infty}\frac{dx}{x^5}=\frac{1}{4}.
\]
Hence, $\lambda={2}\in\sigma(A)$, and
\[
\varphi_{\mu}=(A-{2})A^{-1}\varphi_{\lambda}=\frac{x^2-{2}}{x^{10/3}}, \quad
\psi_{\mu}=(A-{2})A^{-1}\psi_{\lambda}=\frac{x^2-{2}}{x^{11/3}}.
\]
Also from (\ref{f23}) we have that
\[
m_z=(A-\lambda)(A-z)^{-1}\varphi_{\lambda}=\frac{x^2-2}{x^2-z}\frac{1}{x^{4/3}}, \quad
n_z=(A-\bar\lambda)(A-\bar z)^{-1}\psi_{\lambda}=\frac{x^2-2}{x^2-z}\frac{1}{x^{5/3}},
\]
and from  (\ref{f24}) we have $b_z=(\lambda-z)^{-1}(\varphi_{\lambda}, n_{\bar z})^{-1}=
({2}-z)^{-1}(\varphi_{\lambda}, n_{\bar z})^{-1}$,
where
\[
(\varphi_{\lambda}, n_{\bar z})=\int_{1}^{\infty}\frac{1}{x^{4/3}}\frac{x^2-2}{x^2-z}\frac{1}{x^{5/3}}\,dx
=\left(\frac{1}{z^2}-\frac{1}{2z}\right)\ln\sqrt{1-z}+\frac{1}{z}.
\]
Hence (\ref{f88}) has the form
\[
(\tilde A-z)^{-1}=\frac{1}{x^2-z}+b_z\left(\cdot,\frac{x^2-2}{x^2-\bar z}\frac{1}{x^{5/3}}\right)
\frac{x^2-2}{x^2-z}\frac{1}{x^{4/3}}.
\]
Moreover,
\[
\omega_1=\frac{x^2}{x^{5/3}}-\frac{1/2}{1/4}\frac{1}{x^{5/3}}=\frac{x^2-2}{x^{5/3}}, \quad
\omega_2=\frac{x^2}{x^{4/3}}-\frac{1/2}{1/4}\frac{1}{x^{4/3}}=\frac{x^2-2}{x^{4/3}}.
\]
Since $\omega_1,\omega_2\in{\mathcal H}_{-2}\setminus{\mathcal
  H}_{-1}$, $\tilde A\in{\mathcal P}_{\tau}(A)$.

{\bf Example 5}. This example is a modification of the previous one
and illustrates Corollary \ref{c8}.  Let ${\mathcal
  H}=L_2([1,\infty),dx)$ and, as above, $A$ be an operator of
multiplication by $x^2$, namely $Af(x)=x^2f(x)$, $f\in{\mathfrak
  D}(A)$, where ${\mathfrak D}(A):=\{ f(x)\in L_2 \ | \ x^2f(x)\in L_2
\}$.  For simplicity we also put $0=\mu\notin\sigma (A)$, but
$\varphi=\varphi_{\lambda}=x^{-2\frac{1}{3}}$,
$\psi=\psi_{\lambda}=x^{-2\frac{2}{3}}$, $\varphi,\psi\in{\mathcal
  H}_{+1}=L_2([1,\infty ),x^2dx)$.  In particular ${\mathcal
  H}_{+2}=L_2([1,\infty ),x^4dx)$ and $\varphi,\psi\notin {\mathcal
  H}_{+2}$, Then
\[
(\varphi,\psi)=\int_{1}^{\infty}\frac{dx}{x^5}=\frac{1}{4}, \quad
((A-\mu)^{-1}\varphi,\psi)=\int_{1}^{\infty}\frac{dx}{x^7}=\frac{1}{6}.
\]
Hence, $\lambda=3/2\in\sigma(A)$; and
\[
\varphi_{\mu}=(A-3/2)A^{-1}\varphi_{\lambda}=\frac{x^2-3/2}{x^{13/3}}, \quad
\psi_{\mu}=(A-3/2)A^{-1}\psi_{\lambda}=\frac{x^2-3/2}{x^{14/3}}.
\]
And also from (\ref{f23}) we have that
\[
m_z=(A-\lambda)(A-z)^{-1}\varphi_{\lambda}=\frac{x^2-3/2}{x^2-z}\frac{1}{x^{7/3}}, \quad
n_z=(A-\bar\lambda)(A-\bar z)^{-1}\psi_{\lambda}=\frac{x^2-3/2}{x^2-z}\frac{1}{x^{8/3}},
\]
and from  (\ref{f24}) it follows that $b_z=(\lambda-z)^{-1}(\varphi_{\lambda}, n_{\bar z})^{-1}=
(3/2-z)^{-1}(\varphi_{\lambda}, n_{\bar z})^{-1}$,
where
\[
(\varphi_{\lambda}, n_{\bar z})=\int_{1}^{\infty}\frac{1}{x^{7/3}}\frac{x^2-3/2}{x^2-z}\frac{1}{x^{8/3}}\,dx
=\left(\frac{3}{2z^3}-\frac{1}{z^2}\right)\ln\sqrt{1-z}+\left(\frac{3}{4z^2}-\frac{1}{2z}\right) +\frac{3}{8z}.
\]
Hence (\ref{f88}) has the form
\[
(\tilde A-z)^{-1}=\frac{1}{x^2-z}+b_z\left(\cdot,\frac{x^2-3/2}{x^2-\bar z}\frac{1}{x^{8/3}}\right)
\frac{x^2-3/2}{x^2-z}\frac{1}{x^{7/3}}.
\]
Moreover,
\[
\omega_1=\frac{x^2}{x^{8/3}}-\frac{1/4}{1/6}\frac{1}{x^{8/3}}=\frac{x^2-3/2}{x^{8/3}}, \quad
\omega_2=\frac{x^2}{x^{7/3}}-\frac{1/4}{1/6}\frac{1}{x^{7/3}}=\frac{x^2-3/2}{x^{7/3}}.
\]
Since $\omega_1,\omega_2\in{\mathcal H}_{-2}\setminus{\mathcal H}_{-1}$, 
$\tilde A\in{\mathcal P}_{\tau}(A)$, and we can exactly calculate the coupling constant.
\[
\langle\varphi_{\lambda},\omega_1\rangle=\int_{1}^{\infty}\frac{1}{x^{7/3}}\frac{x^2-3/2}{x^{8/3}}\,dx=
\langle\psi_{\lambda}, \omega_2\rangle=\int_{1}^{\infty}\frac{1}{x^{8/3}}\frac{x^2-3/2}{x^{7/3}}\,dx=
1/8.
\]
Due to (\ref{fnd}), (\ref{fnd+}) we have that
$\alpha=-\frac{1}{\langle\varphi_{\lambda},\omega_1\rangle}=-8$.

{\bf Example 6}. This example once more illustrates Theorems \ref{t1},
\ref{t2}, and \ref{t3}.  Let ${\mathcal H}=L_2({\mathbb R}^1)$ and $A$
be also the Laplace operator, namely $Af(x)=-f''(x)$, ${\mathfrak
  D}(A)=W_2^2({\mathbb R}^1)$ is the Sobolev space.  The operator
$A\geq 0$ is positive and has purely absolutely continuous spectrum,
i.e., $\sigma(A)=\sigma_c(A)=[0,\infty)$.  Let us consider the
expression (\ref{f4}) that describes the $\delta$-interaction with
retardation on the real line.

Using \cite{AGHKH} we can write the resolvent of such operators
(corresponding to one parameter family) at a regular point $k^2$,
($\Im k^2>0$), i.e its integral kernel,
\[
(-\tilde\Delta-k^2)^{-1}(x,\xi)=
(i/2k)e^{ik\vert x-\xi\vert}+\alpha(2k)^{-1}(i\alpha+2k)^{-1}e^{ik[\vert x-x_2\vert+\vert x_1-\xi\vert]},
\]
where
\[
{\rm Im}k>0, \quad \alpha\in{\mathbb C}, \quad x,\xi,x_1,x_2 \in {\mathbb R}^1, \quad x_1<x_2.
\]
It is not hard to understand that the essential spectrum is
$\sigma_{ess}(-\tilde\Delta)=\sigma_{ac}(-\Delta)=[0,\infty)$, and,
for the singularly continuous spectrum, we have
$\sigma_{sc}(-\tilde\Delta)=\emptyset$.  Moreover, if ${\Re}\alpha<0$,
then the operators $-\tilde\Delta$ (and $-\tilde\Delta^*$) possess
precisely one negative, simple eigenvalue $\{-\alpha^2/4\}$ (and
$\{-\bar\alpha^2/4\}$) with the corresponding normalized eigenfunction
\[
\varphi=(-\alpha/2)^{1/2}e^{\alpha\vert x-y_1\vert/2}, \quad
\psi=(-\bar\alpha/2)^{1/2}e^{\bar\alpha\vert x-y_2\vert/2}.
\]
It is obvious that $-\tilde \Delta\in{\mathcal P}(-\Delta)$.

{\it Acknowledgments}. The authors are grateful to Professor
L.~P.~Nizhnik for stimulating discussions and providing important
suggestions about an approach of viewing nonsingular perturbations.


\begin{thebibliography}{100000}

\bibitem{AGHKH}
{S.~Albeverio, F.~Gesztesy,  R.~H{\o}egh-Krohn, H.~Holden},   
\textit{Solvable Models in Quantum Mechanics},   
{2nd~ed., With an appendix by Pavel Exner}, 
{AMS Chelsea Publishing},    
{Providence, RI},     
{2005}.     

\bibitem{ADK}
{S.~Albeverio, M.~Dudkin, and V.Koshmanenko},   
\textit{Rank-one singular perturbations with a dual pair of eigenvalues}, 
{Lett. Math. Phys.} 
\textbf{63} 
{(2003)},   
{no.~3},    
{219-228}.


\bibitem{AKu1}
{S.~Albeverio and P.~Kurasov},   
\textit{Singular Perturbations of Differential Operators. Solvable Schr{\"o}dinger Type Operators},  
{London Mathematical Society Lecture Note Series, 271}, 
{Cambridge University Press},    
{Cambridge},     
{2000}.    

\bibitem{AKKN}
{S.~Albeverio,  P.~Kurasov, V.~Koshmanenko, and L.~Nizhnik},   
\textit{On approximations of rank one ${\mathcal H}_{-2}$ perturbations}, 
{Proc. Amer. Math. Soc.} 
\textbf{131} 
{(2002)},   
{no.~5},    
{1443--1452}. 

\bibitem{AKN}
{S.~Albeverio, S.~Kuzhel, and L.~Nizhnik},   
\textit{On the perturbation theorey of self-adjoint operators}, 
{Tokyo Journal of Mathematics} 
\textbf{31} 
{(2008)},   
{no.~2},    
{273--292}. 

\bibitem{AN}
{S.~Albeverio and L.~Nizhnik},   
\textit{Schr\"{o}dinger operators with nonlocal potentials}, 
{Methods Funct. Anal. Topology} 
\textbf{10} 
{(2013)},   
{no.~3},    
{199--210}. 

\bibitem{BB}
{Y.~M.~Berezansky  and J.~Brasche},   
\textit{Generalized selfadjoint operators and their singular perturbations}, 
{Methods Funct. Anal. Topology} 
\textbf{8} 
{(2002)},   
{no.~4},    
{1--14}. 

\bibitem{B3}
{Yu. M. Berezansky, Z.~G. Sheftel, G.~F. Us},   
\textit{Functional Analysis},  
{Vols.~1,~2}, 
{Birkh\"auser Verlag},    
{Basel---Boston---Berlin},     
{1996. (Russian edition: Vyshcha shkola, Kiev, 1990)}

\bibitem{D0}
{M. Dudkin},   
\textit{Invariant symmetric restrictions of a self-adjoint operator. {\rm{II}}}, 
{Ukrain. Mat. Zh.} 
\textbf{50} 
{(1998)},   
{no.~6},    
{781--791}. 
(Ukrainian); English transl.
{Ukrainian Math.~J.} 
\textbf{50} 
{(1998)},   
{no.~6},    
{888--900}.  

\bibitem{D}
{M.~Dudkin},   
\textit{Singularly perturbed normal operators}, 
{Ukrain. Mat. Zh.} 
\textbf{51} 
{(1999)},   
{no.~8},    
{1045--1053}.  
(Ukrainian); English transl.
{Ukrainian Mat.~J.} 
\textbf{51} 
{(1999)},   
{no.~8},    
{1177--1187}.  

\bibitem{Ka}
{T.~Kato},   
\textit{Perturbation Theory for Linear Operators},   
{Die Grundlehren der Mathematischen Wissenschaften, Band 132}, 
{Springer-Verlag New York},    
{Inc., New York},     
{1966}.     


\bibitem{Ka2}
{T.~Kato},   
\textit{Wave operators and similarity for some non-selfadjoint operators}, 
{Math. Annalen}
\textbf{162} 
{(1966)},   
{258--279}. 

\bibitem{Ko17}
{V.~Koshmanenko},   
\textit{Singular Quadratic Forms in Perturbation Theory},  
{Mathematics and its Applications, Vol.~474}, 
{Kluwer Academic Publishers},    
{Dordrecht--Boston--London},  
{1999}. (Russian edition: Naukova Dumka, Kyiv, 1993)     



\bibitem{Ko15}
{V.~Koshmanenko},   
\textit{Singularly perturbed operators. Mathematical results in quantum mechanics} 
{(Blossin, 1993)},
{Oper. Theory Adv. Appl., Vol.~70}, 
{Birkhauser, Basel},
{1994, pp. 347--351}.  


\bibitem{L}
{V.~B.~Lidskii},   
\textit{The non-self-adjoint operator of Sturm-Liouville type with discrete spectrum}, 
{Trudy Mosk. Mat. Obsh.} 
\textbf{9} 
{(1960)},   
{45--79. (Russian)} 




\bibitem{MM}
{M.~M.~Malamud and V.~I.~Mogilevskii},   
\textit{Kre\u{\i}n type formula for canonical resolvents of dual pairs of linear relations}, 
{Methods Funct. Anal. Topology} 
\textbf{8} 
{(2002)},   
{no.~4},    
{72--100}. 

\bibitem{N1}
{L.~P.~Nizhnik},   
\textit{Inverse nonlocal Sturm-Liouville problem}, 
{Inverse Problems} 
\textbf{26} 
{(2010)},   
{125006 (9 pp)}. 
doi:10.1088/0266-5611/26/12/125006


\bibitem{N2}
{L.~P.~Nizhnik},   
\textit{Inverse spectral  nonlocal problem for the first order ordinary differential equation}, 
{Tamkang Journal of Mathematics} 
\textbf{42} 
{(2011)},   
{no.~3},    
{385--394}. 


\bibitem{V}
{M.~I.~Vishik},   
\textit{On general boundary-value problems for elliptic differential equation}, 
{Trudy Mosk. Mat. Obsh.} 
\textbf{1} 
{(1952)},   
{187--246. (Russian)} 

\bibitem{VD}
{T.~Vdovenko and M.~Dudkin},   
\textit{Singular rank one nonsymmetric perturbations of a selfadjoint operator}, 
{Proceedings of Institute of Mathematics of NAS of Ukraine}, 
\textbf{12} 
{(2015)},   
{no.~1},    
{57--73. (Ukrainian)}  



\end{thebibliography}
\end{document}